\documentclass[10pt]{amsart}
\usepackage{amsmath}
\usepackage{amssymb} 

\def\beqnn{\begin{eqnarray*}}\def\eeqnn{\end{eqnarray*}}
\newtheorem{theorem}{Theorem}[section]

\newtheorem{proposition}[theorem]{Proposition}
\newtheorem{corollary}[theorem]{Corollary}
\newtheorem{lemma}[theorem]{Lemma}

\theoremstyle{definition}

\newtheorem{remark}[theorem]{Remark}

\theoremstyle{question}
\newtheorem{question}[theorem]{Question}

\numberwithin{equation}{section}

\begin{document}

\title{On the Hilbert-type operators acting from function spaces into sequence spaces}


\author{Jianjun Jin}
\address{School of Mathematics Sciences, Hefei University of Technology, Xuancheng Campus, Xuancheng 242000, P.R.China}
\email{ jin@hfut.edu.cn, jinjjhb@163.com}


\thanks{The author was supported by National Natural Science Foundation of China(Grant Nos. 11501157). }






\begin{abstract}
In this paper we introduce and study some Hilbert-type operators acting from the function spaces into the sequence spaces.  We give some sufficient and necessary conditions for the boundedness and compactness of these Hilbert-type operators.  Also, for some special cases, we obtain the sharp estimates for the norms of certain Hilbert-type operators.
\end{abstract}

\maketitle
\section{Introductions and main results}

In this paper, for two positive numbers $A, B$, we write $A \preceq B$, or $A \succeq B$, if there exists a positive constant $C$  independent of the arguments such that $A \leq C B$, or $A \geq C B$, respectively. We will write $A \asymp B$ if both $A \preceq B$ and $A \succeq B$.

Let $p>1$.  We denote by $q$ the conjugate index of $p$, i.e., $\frac{1}{p}+\frac{1}{q}=1$. We denote  the interval $(0, +\infty)$ by $\mathbb{R}_{+}$. Let $\mathcal{M}(\mathbb{R}_{+})$ be the class of all measurable functions on $\mathbb{R}_{+}$. The usual Lesbegue space of measurable functions on $\mathbb{R}_{+}$, denoted by  $L^p(\mathbb{R}_{+})$,  is defined as
\begin{equation*}L^{p}(\mathbb{R}_{+}):=\{f\in \mathcal{M}(\mathbb{R}_{+}): \|f\|_{p}=(\int_{\mathbb{R}_{+}} |f(x)|^p\,dx)^{\frac{1}{p}}<+\infty\}.\end{equation*}

We denote by $l^{p}$ the Lesbegue space of infinite sequences, i.e.,
\begin{equation*}l^{p}:=\{a=\{a_n\}_{n=1}^{\infty}: \|a\|_{p}=(\sum_{n=1}^{\infty} |a_{n}|^p)^{\frac{1}{p}}<+\infty \}.\end{equation*}
The classical Hilbert operator $H$, induced by Hilbert kernel $\frac{1}{x+y}$, is defined as
\begin{equation*} H(f)(y):=\int_{\mathbb{R}_{+}} \frac{f(x)}{x+y}\,dx, \: f\in \mathcal{M}(\mathbb{R}_{+}), \: y\in \mathbb{R}_{+}.\end{equation*}

The Hilbert series operator, denoted by $\widetilde{H}$, is similarly defined as
\begin{equation*} \widetilde{H}(a)(n)=\sum_{n=1}^\infty \frac{a_m}{m+n}, \: a=\{a_m\}_{m=1}^{\infty}, \: n\in \mathbb{N}.
\end{equation*}

It is well known that
\begin{proposition}\label{h-1}
Let $p>1$. Then $H$($\widetilde{H}$) is bounded on $L^{p}(\mathbb{R}_{+})$($l^{p}$), respectively. Moreover, the norm $\|H\|$($\|\widetilde{H}\|$) of
$H$($\widetilde{H}$) is $\pi \csc{\frac{\pi}{p}}$, respectively, where
$$\|H\|:=\sup_{f(\neq \theta)\in L^p(\mathbb{R}_{+})}\frac{\|Hf\|_{p}}{\|f\|_{p}},\quad \|\widetilde{H}\|:=\sup_{a(\neq \theta)\in l^p}\frac{\|\widetilde{H}a\|_{p}}{\|a\|_p}. $$
\end{proposition}

Recently, many attentions have been paid to the study of the boundedness and estimates for the norms of the following type operator $H_K$, which is defined as
\begin{equation*} H_{K}(f)(n):=\int_{\mathbb{R}_{+}} f(x)K(x ,n)\,dx, \: f\in \mathcal{M}(\mathbb{R}_{+}), \: n\in \mathbb{N},\end{equation*}
where $K(x,y)$ is a non-negative function on $\mathbb{R}_{+}\times \mathbb{R}_{+}$. For more detailed  introductions to  this topic, see the	book \cite{YL} of Yang and Lokenath,  and the references cited therein.  It should be pointed out that Hardy et. al. first considered this topic in \cite{HLP} (Theorem 351).

In particular, take $K(x, y)=\frac{1}{x+y}$, we get the Hilbert-type operator $\widehat{H}$ as follows.
$$\widehat{H}(f)(n):=\int_{\mathbb{R}_{+}}\frac{f(x)}{x+n}dx,\, f\in   \mathcal{M}(\mathbb{R}_{+}), n\in \mathbb{N}.$$

In this paper, we continue to study this topic. By introducing some parameters,  we define the following Hilbert-type operator $\mathbf{H}_{\theta_1, \theta_2, \lambda}^{\alpha, \beta}$ as
$$\mathbf{H}_{\theta_1, \theta_2, \lambda}^{\alpha, \beta}(f)(n):=n^{\frac{1}{p}[(\theta_2-1)+\beta\theta_2]}\int_{\mathbb{R}_{+}}\frac{x^{\frac{1}{q}[(\theta_1-1)+\alpha\theta_1]}f(x)}{x^{\alpha \theta_1}(x^{\theta_1}+n^{\theta_2})^{\lambda}}dx,\, f\in  \mathcal{M}(\mathbb{R}_{+}),\; n\in \mathbb{N},$$
where $\lambda>0$, $0<\theta_1, \theta_2 \leq 1$, $-1<\alpha, \beta <p-1$.

When $\theta_1=\theta_2=\lambda=1$, $\alpha=\beta=0$, the operator $\mathbf{H}_{\theta_1, \theta_2, \lambda}^{\alpha, \beta}$ reduces to the operator $\widehat{H}$.  We first study the boundedness of $\mathbf{H}_{\theta_1, \theta_2, \lambda}^{\alpha, \beta}$. We will prove that

\begin{theorem}\label{m-1-1}
Let $p>1$, $\lambda>0$, $0<\theta_1, \theta_2 \leq 1$, $-1<\alpha, \beta <p-1$,  and $\mathbf{H}_{\theta_1, \theta_2, \lambda}^{\alpha, \beta}$ be defined as obove. Then  $\mathbf{H}_{\theta_1, \theta_2, \lambda}^{\alpha, \beta}$ is bounded from  $L^p(\mathbb{R}_{+})$ into $l^p$  if and only if $\lambda \geq 1+\frac{1}{p}(\beta-\alpha).$\end{theorem}

Next, we consider the case when $\lambda= 1+\frac{1}{p}(\beta-\alpha).$   In this case, we denote by $\widetilde{\mathbf{H}}_{\theta_1, \theta_2}^{\alpha, \beta}$ the operator instead of $\mathbf{H}_{\theta_1, \theta_2, \lambda}^{\alpha, \beta}$.  That is
$$\widetilde{\mathbf{H}}_{\theta_1, \theta_2}^{\alpha, \beta}(f)(n):=n^{\frac{1}{p}[(\theta_2-1)+\beta\theta_2]}\int_{\mathbb{R}_{+}}\frac{x^{\frac{1}{q}[(\theta_1-1)+\alpha\theta_1]}f(x)}{x^{\alpha \theta_1}(x^{\theta_1}+n^{\theta_2})^{1+\frac{1}{p}(\beta-\alpha)}}dx,\, f\in  \mathcal{M}(\mathbb{R}_{+}),\; n\in \mathbb{N}.$$

We denote by  $\|\widetilde{\mathbf{H}}_{\theta_1, \theta_2}^{\alpha, \beta}\|$ the norm of $\widetilde{\mathbf{H}}_{\theta_1, \theta_2}^{\alpha, \beta}$. We will show that
\begin{theorem}\label{m-1-2}
Let $p>1$, $0<\theta_1, \theta_2 \leq 1$, $-1<\alpha, \beta <p-1$, and $\widetilde{\mathbf{H}}_{\theta_1, \theta_2}^{\alpha, \beta}$ be defined as above.  Then  $\widetilde{\mathbf{H}}_{\theta_1, \theta_2}^{\alpha, \beta}$ is bounded from  $L^p(\mathbb{R}_{+})$ into $l^p$ and
\begin{equation}\label{norm}\|\widetilde{\mathbf{H}}_{\theta_1, \theta_2}^{\alpha, \beta}\|= \frac{1}{\theta_2^{\frac{1}{p}}\theta_1^{\frac{1}{q}}}B(\frac{1+\beta}{p}, \frac{p-1-\alpha}{p}).\end{equation}  \end{theorem}
Here $B(\cdot, \cdot)$ is the well-known Beta function, which is defined as
$$B(u,v)=\int_{0}^{\infty}\frac{t^{u-1}}{(1+t)^{u+v}}\,dt,\: u>0,v>0.$$
It is known that
$$B(u,v)=\int_{0}^{1}t^{u-1}(1-t)^{v-1}\,dt=\frac{\Gamma(u)\Gamma{(v)}}{\Gamma(u+v)},$$
where $\Gamma(\cdot) $ is the Gamma function, defined as
$$\Gamma(x)=\int_{0}^{\infty}e^{-t} t^{x-1}\,dt,\: x>0.$$
For more infromations to these special functions, see \cite{AAR}.

When $\theta_1=\theta_2=1$,  we see from Theorem \ref{m-1-1} that
${\mathbf{H}}_{1, 1, \lambda}^{\alpha, \beta}$ is not bounded from  $L^p(\mathbb{R}_{+})$ into $l^p$ if $\lambda<1+\frac{1}{p}(\beta-\alpha)$.  On the other hand, we notice that, for $\lambda>0$,
$$\int_{[0, 1)}t^{x+n-1}(1-t)^{\lambda-1}dt=\frac{\Gamma(x+n)\Gamma(\lambda)}{\Gamma(x+n+\lambda)}, x>0, n\geq 1 .$$

Then, by using the fact
\begin{equation}\label{g}
\Gamma(x) = \sqrt{2\pi} x^{x-\frac{1}{2}}e^{-x}[1+r(x)],\, |r(x)|\leq e^{\frac{1}{12x}}-1,\, x>0,\end{equation}
we see that
$$\int_{[0, 1)}t^{x+n-1}(1-t)^{\lambda-1}dt \asymp  \frac{1}{(x+n)^{\lambda}}.$$

Hence, in order for ${\mathbf{H}}_{1,1,\lambda}^{\alpha, \beta}: L^p(\mathbb{R}_{+}) \rightarrow l^p$ to be bounded  when $\lambda<1+\frac{1}{p}(\beta-\alpha)$,  let $\mu$ be a positive Borel measure on $[0, 1)$, we will consider the operator
$\widehat{\mathbf{H}}^{\alpha, \beta}_{\lambda, \mu}$, which is defined as
$$\widehat{\mathbf{H}}^{\alpha, \beta}_{\lambda, \mu}(f)(n):=n^{\frac{\beta}{p}}\int_{\mathbb{R}_{+}}x^{-\frac{\alpha}{p}}\mu_{\lambda}[x+n]f(x)dx,\, f\in  \mathcal{M}(\mathbb{R}_{+}), \, n\in \mathbb{N}.$$
Where
\begin{equation}\label{mu}\mu_{\lambda}[z]:=\int_{[0, 1)}t^{z-1}(1-t)^{\lambda-1}d\mu(t), \, z, \lambda >0.\end{equation}

We then study the problem of characterizing measures $\mu$ such that $\widehat{\mathbf{H}}^{\alpha, \beta}_{\lambda, \mu} : L^p(\mathbb{R}_{+})\rightarrow l^p$ is bounded. We provide a sufficient and necessary condition of $\mu$ for which $\widehat{\mathbf{H}}^{\alpha, \beta}_{\lambda, \mu}$ is bounded from $L^p(\mathbb{R}_{+})$ into $l^p$.  We will prove the following
\begin{theorem}\label{m-1-3}
Let $p>1, \lambda>0$,  $-1<\alpha, \beta<p-1$.  Let $\mu$ be a positive Borel measure on $[0, 1)$ such that $d\nu(t):=(1-t)^{\lambda-1}d\mu(t)$ is a finite measure on $[0, 1)$, and  $\widehat{\mathbf{H}}^{\alpha, \beta}_{\lambda, \mu}$ be defined as above. Then $\widehat{\mathbf{H}}^{\alpha, \beta}_{\lambda, \mu} : L^p(\mathbb{R}_{+})\rightarrow l^p$ is bounded if and only if
$\nu$ is a $[1+\frac{1}{p}(\beta-\alpha)]$-Carleson measure on $[0, 1)$.
\end{theorem}

Here,  for $s>0$, a positive Borel measure $\mu$ on $[0,1)$, we say $\mu$ is an $s$-Carleson measure if there is a constant $C>0$ such that $$\mu([t, 1))\leq C (1-t)^s$$ holds for all $t\in [0, 1)$.

We also characterize measures $\mu$ such that $\widehat{\mathbf{H}}^{\alpha, \beta}_{\lambda, \mu} : L^p(\mathbb{R}_{+})\rightarrow l^p$ is compact. We shall show that
 \begin{theorem}\label{m-1-4}
Let $p>1, \lambda>0$,  $-1<\alpha, \beta<p-1$.  Let $\mu$ be a positive Borel measure on $[0, 1)$ such that $d\nu(t):=(1-t)^{\lambda-1}d\mu(t)$ is a finite measure on $[0, 1)$,  and $\widehat{\mathbf{H}}^{\alpha, \beta}_{\lambda, \mu}$ be defined as above. Then $\widehat{\mathbf{H}}^{\alpha, \beta}_{\lambda, \mu} : L^p(\mathbb{R}_{+})\rightarrow l^p$ is compact if and only if
$\nu$ is a vanishing $[1+\frac{1}{p}(\beta-\alpha)]$-Carleson measure on $[0, 1)$.
\end{theorem}

Here, an $s$-Carleson measure $\mu$ on $[0, 1)$  is  said to be a vanishing $s$-Carleson measure, if  it satisfies further that
$$\lim_{t\rightarrow 1^{-}}\frac{\mu([t, 1))}{(1-t)^s}=0.$$

The paper is organized as follows. We will first prove Theorem \ref{m-1-2} in the next section. The proof of Theorem \ref{m-1-1} will be give in Section 3.  We prove Theorem \ref{m-1-3} and \ref{m-1-4} in Section 4.  Final remarks will be present in Section 5.


\section{Proof of Theorem \ref{m-1-2}}
For $p>1$, $0<\theta_1, \theta_2 \leq 1$, $-1<\alpha, \beta <p-1$, we set
$$w_1(n):=\int_{0}^{\infty}\frac{x^{\theta_1-1}}{(x^{\theta_1}+n^{\theta_2})^{1+\frac{1}{p}(\beta-\alpha)}} \cdot \frac{n^{\frac{\theta_2(1+\beta)}{p}}}{x^{\frac{\theta_1(1+\alpha)}{p}}}\,dx, \,\, n\in \mathbb{N};$$
$$w_2(x):=\sum_{n=1}^{\infty}\frac{n^{\theta_2-1}}{(x^{\theta_1}+n^{\theta_2})^{1+\frac{1}{p}(\beta-\alpha)}} \cdot \frac{x^{\frac{\theta_1(p-1-\alpha)}{p}}}{n^{\frac{\theta_2(p-1-\beta)}{p}}}, \,\, x>0.$$

Then, by a change of variables, we have
\begin{eqnarray}\label{w-1-1}w_1(n)&=& \int_{0}^{\infty}\frac{s^{-\frac{1+\alpha}{p}}}{(s+n^{\theta_2})^{1+\frac{1}{p}(\beta-\alpha)}} \cdot n^{\frac{\theta_2(1+\beta)}{p}}\,dx  \nonumber \\
&= &\frac{1}{\theta_1} \int_{0}^{\infty}\frac{t^{-\frac{1+\alpha}{p}}}{(1+t)^{1+\frac{1}{p}(\beta-\alpha)}}\,dt=\frac{1}{\theta_1}  B(\frac{1+\beta}{p}, \frac{p-1-\alpha}{p}).
\end{eqnarray}

On the other hand, we have
\begin{eqnarray}\label{w-2-2}w_2(x)&\leq &\int_{0}^{\infty}\frac{u^{\theta_2-1}}{(x^{\theta_1}+u^{\theta_2})^{1+\frac{1}{p}(\beta-\alpha)}} \cdot \frac{x^{\frac{\theta_1(p-1-\alpha)}{p}}}{u^{\frac{\theta_2(p-1-\beta)}{p}}}\,du \nonumber \\
&= &\frac{1}{\theta_2} \int_{0}^{\infty}\frac{t^{-\frac{p-1-\beta}{p}}}{(1+t)^{1+\frac{1}{p}(\beta-\alpha)}}\,dt=\frac{1}{\theta_2}  B(\frac{1+\beta}{p}, \frac{p-1-\alpha}{p}).
\end{eqnarray}
Here, we have used the change of variables $t=u^{\theta_2}/x^{\theta_1}$.

We start to prove Theorem \ref{m-1-2}.  For $f \in L^p(\mathbb{R}_{+}), n\in \mathbb{N}$,  we write
\begin{eqnarray}\lefteqn{n^{\frac{1}{p}[(\theta_2-1)+\beta\theta_2]}\left|\int_{\mathbb{R}_{+}}\frac{x^{\frac{1}{q}[(\theta_1-1)+\alpha\theta_1]}f(x)}{x^{\alpha \theta_1}(x^{\theta_1}+n^{\theta_2})^{1+\frac{1}{p}(\beta-\alpha)}}dx\right|}\nonumber  \\ &&\leq n^{\frac{1}{p}[(\theta_2-1)+\beta\theta_2]}\int_{\mathbb{R}_{+}}\frac{x^{\frac{1}{q}[(\theta_1-1)+\alpha\theta_1]}|f(x)|}{x^{\alpha \theta_1}(x^{\theta_1}+n^{\theta_2})^{1+\frac{1}{p}(\beta-\alpha)}}dx\nonumber \\ &&=\int_{\mathbb{R}_{+}} \left\{[V(x,n)]^{\frac{1}{p}}E_1(x, n)\cdot [V(x,n)]^{\frac{1}{q}}E_2(x, n)\right\}\,dx:=I(n), \nonumber \end{eqnarray}
where
\begin{equation*}
{V}(x, n)=\frac{1}{(x^{\theta_1}+n^{\theta_2})^{1+\frac{1}{p}(\beta-\alpha)}};
\end{equation*}

\begin{equation*}
{E}_1(x, n)=
\frac{x^{\frac{\theta_1(1+\alpha)}{pq}-\frac{\theta_1 \alpha}{p}}}
{n^{\frac{\theta_2(p-1-\beta)}{p^2}}}
\cdot
n^{\frac{1}{p}(\theta_2-1)}\cdot |f(x)|;
\end{equation*}

\begin{equation*}
{E}_2(x, n)=
\frac{n^{\frac{\theta_2(p-1-\beta)}{p^2}+\frac{\theta_2\beta}{p}}}
{x^{\frac{\theta_1(1+\alpha)}{pq}}}
\cdot x^{\frac{1}{q}(\theta_1-1)}.
\end{equation*}

Applying the H\"{o}lder's inequality on $I(n)$, we get from (\ref{w-1-1}) that
\begin{eqnarray}
I(n) &\leq & \left[\int_{\mathbb{R}_{+}}{V}(x, n)[{E}_1(x, n)]^p\,dx \right]^{\frac{1}{p}}\left[\int_{\mathbb{R}_{+}}{V}(x, n)[{E}_2(x, n)]^{q} \,dx\right]^{\frac{1}{q}} \nonumber \\
& =& [w_1(n)]^{\frac{1}{q}}\left[\int_{\mathbb{R}_{+}}{V}(x, n)[{E}_1(x, n)]^p\,dx \right]^{\frac{1}{p}}. \nonumber
\end{eqnarray}
It follows from  (\ref{w-2-2}) that
\begin{eqnarray}
\lefteqn{\|\widetilde{\mathbf{H}}_{\theta_1, \theta_2}^{\alpha, \beta}f\|_p=[\sum_{n=1}^{\infty}I^p(n) ]^{\frac{1}{p}}}  \nonumber \\
& \leq&  \frac{1}{\theta_1^{\frac{1}{q}}}[B(\frac{1+\beta}{p}, \frac{p-1-\alpha}{p})]^{\frac{1}{q}}\left[\sum_{n=1}^{\infty}\int_{\mathbb{R}_{+}}{V}(x, n)[{E}_1(x, n)]^p \,dx\right]^{\frac{1}{p}} \nonumber \\
&=& \frac{1}{\theta_1^{\frac{1}{q}}}[B(\frac{1+\beta}{p}, \frac{p-1-\alpha}{p})]^{\frac{1}{q}}\left[\int_{\mathbb{R}_{+}}w_2(x) |f(x)|^p\,dx \right]^{\frac{1}{p}} \nonumber \\
& \leq&  \frac{1}{\theta_2^{\frac{1}{p}}\theta_1^{\frac{1}{q}}}B(\frac{1+\beta}{p}, \frac{p-1-\alpha}{p})\|f\|_p. \nonumber
\end{eqnarray}

This means that $\widetilde{\mathbf{H}}_{\theta_1, \theta_2}^{\alpha, \beta}$ is bounded from $L^p(\mathbb{R}_{+})$ into $l^p$, and
\begin{equation}\label{low}\|\widetilde{\mathbf{H}}_{\theta_1, \theta_2}^{\alpha, \beta}\|\leq \frac{1}{\theta_2^{\frac{1}{p}}\theta_1^{\frac{1}{q}}}B(\frac{1+\beta}{p}, \frac{p-1-\alpha}{p}).\end{equation}

For $\varepsilon>0$, we define $\widetilde{f}(x)=0$, if $x\in (0,1)$; $\widetilde{f}(x)=\varepsilon^{\frac{1}{p}} x^{-\frac{1+\theta_1\varepsilon}{p}}$, if $x\in [1, \infty)$.  Then, we easily see that $\|\widetilde{f}\|_p^p=\theta_1^{-1}.$

We write
\begin{eqnarray}\label{shar-1}
\|\widetilde{\mathbf{H}}_{\theta_1, \theta_2}^{\alpha, \beta}\widetilde{f}\|_p^p=\varepsilon \sum_{n=1}^{\infty}n^{(1+\beta)\theta_2-1} \cdot [J(n)]^p.
\end{eqnarray}
Here $$J(n):=\int_{\mathbb{R}_{+}}\frac{x^{\frac{1}{q}[(\theta_1-1)-(q-1)\alpha\theta_1]}\cdot x^{-\frac{1+\varepsilon\theta_1}{p}}}{(x^{\theta_1}+n^{\theta_2})^{1+\frac{1}{p}(\beta-\alpha)}}\, dx .$$

On the other hand, a calculation yields that
\begin{eqnarray}\label{shar-2}
J(n)=\frac{1}{\theta_1}n^{-\frac{\theta_2}{p}(1+\beta+\varepsilon)} \int_{\frac{1}{n^{\theta_2}}}^{\infty} \frac{t^{-\frac{1+\alpha+\varepsilon}{p}}}{(1+t)^{1+\frac{1}{p}(\beta-\alpha)}}\,dt.
\end{eqnarray}

Also, when $\varepsilon< p-1-\alpha$, we have
\begin{eqnarray}\label{shar-3}
\lefteqn{ \int_{\frac{1}{n^{\theta_2}}}^{\infty} \frac{t^{-\frac{1+\alpha+\varepsilon}{p}}}{(1+t)^{1+\frac{1}{p}(\beta-\alpha)}}\,dt } \nonumber \\
&&= \int_{0}^{\infty} \frac{t^{-\frac{1+\alpha+\varepsilon}{p}}}{(1+t)^{1+\frac{1}{p}(\beta-\alpha)}}\,dt-
 \int_{0}^{\frac{1}{n^{\theta_2}}} \frac{t^{-\frac{1+\alpha+\varepsilon}{p}}}{(1+t)^{1+\frac{1}{p}(\beta-\alpha)}}\,dt  \nonumber  \\
 &&= B(\frac{1+\beta}{p}+\frac{\varepsilon}{p}, \frac{p-1-\alpha}{p}-\frac{\varepsilon}{p})- \int_{0}^{\frac{1}{n^{\theta_2}}}  \frac{t^{-\frac{1+\alpha+\varepsilon}{p}}}{(1+t)^{1+\frac{1}{p}(\beta-\alpha)}}\,dt  \nonumber
 \\&&:= L(\varepsilon)-U(n).
\end{eqnarray}

Combine (\ref{shar-1}), (\ref{shar-2}) and (\ref{shar-3}), we get that
\begin{eqnarray}\label{mo-1}
\|\widetilde{\mathbf{H}}_{\theta_1, \theta_2}^{\alpha, \beta}\widetilde{f}\|_p^p&\geq &\frac{\varepsilon}{\theta_1^p}\sum_{n=1}^{\infty}n^{-1-\varepsilon\theta_2}\cdot [L(\varepsilon)-U(n)]^p.
\end{eqnarray}
By using the Bernoulli's inequality(see \cite{M}), we obtain that
\begin{eqnarray}\label{mo-2}
[L(\varepsilon)-U(n)]^p \geq  [L(\varepsilon)]^p \left [1-\frac{p}{L(\varepsilon)} \int_{0}^{\frac{1}{n^{\theta_2}}}  \frac{t^{-\frac{1+\alpha+\varepsilon}{p}}}{(1+t)^{1+\frac{1}{p}(\beta-\alpha)}}\,dt\right].
\end{eqnarray}

We also note that
\begin{eqnarray}\label{mo-3}
\varepsilon \sum_{n=1}^{\infty}n^{-1-\varepsilon\theta_2}=\frac{1}{\theta_2}[1+o(1)], \,\, \,\, \varepsilon \rightarrow 0^{+},
\end{eqnarray}
and, for $\varepsilon<p-1-\nu$,
\begin{eqnarray}\label{mo-4}
\lefteqn{\sum_{n=1}^{\infty}n^{-1-\varepsilon\theta_2}\int_{0}^{\frac{1}{n^{\theta_2}}}  \frac{t^{-\frac{1+\alpha+\varepsilon}{p}}}{(1+t)^{1+\frac{1}{p}(\beta-\alpha)}}\,dt}\nonumber \\
&\leq& \sum_{n=1}^{\infty}n^{-1-\varepsilon\theta_2}\int_{0}^{\frac{1}{n^{\theta_2}}}  t^{-\frac{1+\alpha+\varepsilon}{p}}\,dt\nonumber  \\
&=& \frac{1}{p-1-\alpha-\varepsilon} \sum_{n=1}^{\infty}n^{-1-\theta_2(\frac{p-1-\alpha}{p}+\frac{\varepsilon}{q})} =O(1), \,\, \varepsilon \rightarrow 0^{+}.
\end{eqnarray}

It follows from (\ref{mo-1})-(\ref{mo-4}) that
\begin{eqnarray}
\|\widetilde{\mathbf{H}}_{\theta_1, \theta_2}^{\alpha, \beta}\widetilde{f}\|_p^p&\geq &\frac{1}{\theta_2\theta_1^p}[1+o(1)]\cdot [L(\varepsilon)]^p\cdot [1-\varepsilon O(1)]. \nonumber
\end{eqnarray}
Hence, we get that
\begin{eqnarray}
\|\widetilde{\mathbf{H}}_{\theta_1, \theta_2}^{\alpha, \beta}\| \geq
\frac{\|\widetilde{\mathbf{H}}_{\theta_1, \theta_2}^{\alpha, \beta}\widetilde{f}\|_p}{\|\widetilde{f}\|_p}\geq \frac{1}{\theta_2^{\frac{1}{p}}\theta_1^{\frac{1}{q}}}[1+o(1)]^{\frac{1}{p}}\cdot [L(\varepsilon)]\cdot [1-\varepsilon O(1)]^{\frac{1}{p}}. \nonumber
\end{eqnarray}

Take $\varepsilon \rightarrow 0^{+}$, we see that
\begin{equation}\label{big}
\|\widetilde{\mathbf{H}}_{\theta_1, \theta_2}^{\alpha, \beta}\|  \geq \frac{1}{\theta_2^{\frac{1}{p}}\theta_1^{\frac{1}{q}}}B(\frac{1+\beta}{p}, \frac{p-1-\alpha}{p}).\end{equation}
Combine (\ref{low}) and (\ref{big}), we see that (\ref{norm}) is true and this proves Theorem \ref{m-1-2}.

\section{Proof of Theorem \ref{m-1-1}}

We first prove the "if" part.  If $\lambda\geq 1+\frac{1}{p}(\beta-\alpha)$,  then, for $f\in \mathcal{M}(\mathbb{R}_{+}), \,  n\in \mathbb{N}$,  it is easy to see that
$$\left|\int_{\mathbb{R}_{+}}\frac{x^{\frac{1}{q}[(\theta_1-1)+\alpha\theta_1]}f(x)}{x^{\alpha \theta_1}(x^{\theta_1}+n^{\theta_2})^{\lambda}}dx\right|\leq
\int_{\mathbb{R}_{+}}\frac{x^{\frac{1}{q}[(\theta_1-1)+\alpha\theta_1]}|f(x)|}{x^{\alpha \theta_1}(x^{\theta_1}+n^{\theta_2})^{1+\frac{1}{p}(\beta-\alpha)}}dx.$$

Consequently, in view of the boundedness of $\widetilde{\mathbf{H}}_{\theta_1, \theta_2}^{\alpha, \beta}$ from $L^p(\mathbb{R}_{+})$ into $l^p$, we conclude that $\mathbf{H}_{\theta_1, \theta_2, \lambda}^{\alpha, \beta}$ is bounded from $L^p(\mathbb{R}_{+})$ into $l^p$ when $\lambda\geq 1+\frac{1}{p}(\beta-\alpha)$. The "if" part is proved.

Next, we prove the "only if" part.  We will show that, if $\lambda<1+\frac{1}{p}(\beta-\alpha)$, then  $\mathbf{H}_{\theta_1, \theta_2, \lambda}^{\alpha, \beta}$ is not bounded from $L^p(\mathbb{R}_{+})$ into $l^p$.

For $\varepsilon>0$, we take, as is done in the proof of Theorem \ref{m-1-2},  $\widetilde{f}(x)=0$, if $x\in (0,1)$; $\widetilde{f}(x)=\varepsilon^{\frac{1}{p}} x^{-\frac{1+\theta_1\varepsilon}{p}}$, if $x\in [1, \infty)$.  Then, we have $\|\widetilde{f}\|_p^p=\theta_1^{-1}.$

Hence we obtain that
\begin{eqnarray}\|\mathbf{H}_{\theta_1, \theta_2, \lambda}^{\alpha, \beta} \widetilde{f}\|_{p}^p&=&\sum_{n=1}^{\infty} n^{(\theta_2-1)+\beta\theta_2} \left[\int_{1}^{\infty}
\frac{x^{\frac{1}{q}[(\theta_1-1)+\alpha\theta_1]}\cdot x^{-\frac{1+\theta_1\varepsilon}{p}}}{x^{\alpha\theta_1}(x^{\theta_1}+n^{\theta_2})^{\lambda}}\,dx\right]^p \nonumber \\ &=&
\sum_{n=1}^{\infty} n^{(\theta_2-1)+\beta\theta_2} \left[\int_{1}^{\infty} \frac{s^{-\frac{1+\alpha+\varepsilon}{p}}}{(s+n^{\theta_2})^{\lambda}}\,ds\right]^p \nonumber  \\ &=&
\sum_{n=1}^{\infty} n^{-1-p\theta_{2}[\lambda-1-\frac{1}{p}(\beta-\alpha)+\frac{\varepsilon}{p}]}\left[ \int_{\frac{1}{n^{\theta_2}}}^{\infty} \frac{t^{-\frac{1+\alpha+\varepsilon}{p}}}{(1+t)^{\lambda}}\,dt\right]^{p} \nonumber \\
&\geq & \sum_{n=1}^{\infty} n^{-1-p\theta_{2}[\lambda-1-\frac{1}{p}(\beta-\alpha)+\frac{\varepsilon}{p}]} \left[\int_{1}^{\infty} \frac{t^{-\frac{1+\alpha+\varepsilon}{p}}}{(1+t)^{\lambda}}\,dt \right]^{p}.\nonumber
\end{eqnarray}

We suppose  $\mathbf{H}_{\theta_1, \theta_2, \lambda}^{\alpha, \beta}:  L^p(\mathbb{R}_{+}) \rightarrow l^p$ is bounded. Then there exists a constant $C_1>0$ such that
\begin{eqnarray}\label{c-1}
\quad\quad C_1&\geq&\frac{\|\mathbf{H}_{\theta_1, \theta_2, \lambda}^{\alpha, \beta} \widetilde{f}\|_{p}^p}{\|\widetilde{f}\|_p^p}\nonumber \\&\geq&\theta_1\sum_{n=1}^{\infty} n^{-1+p\theta_{2}[1-\lambda+\frac{1}{p}(\beta-\alpha)-\frac{\varepsilon}{p}]} \left[\int_{1}^{\infty} \frac{t^{-\frac{1+\alpha+\varepsilon}{p}}}{(1+t)^{\lambda}}\,dt \right]^{p}.
\end{eqnarray}

But, if $\lambda<1+\frac{1}{p}(\beta-\alpha)$, then, when $\varepsilon<p[(1-\lambda)+\frac{1}{p}(\beta-\alpha)]$, we see from $p\theta_{2}[1-\lambda+\frac{1}{p}(\beta-\alpha)-\frac{\varepsilon}{p}]:=\delta>0$ that
$$\sum_{n=1}^{\infty} n^{-1+p\theta_{2}[1-\lambda+\frac{1}{p}(\beta-\alpha)-\frac{\varepsilon}{p}]}  =\sum_{n=1}^{\infty} n^{-1+\delta}=+\infty.$$
This means that (\ref{c-1}) is a contradiction. This proves that $\mathbf{H}_{\theta_1, \theta_2, \lambda}^{\alpha, \beta}$ is not bounded from $L^p(\mathbb{R}_{+})$ into $l^p$ if $\lambda<1+\frac{1}{p}(\beta-\alpha)$. The theorem is proved.

\section{Proofs of Theorem \ref{m-1-3} and \ref{m-1-4}}
In the proofs of Theorem \ref{m-1-3} and \ref{m-1-4}, we need the following
\begin{lemma}\label{lem}
Let $\lambda>0, -1<\alpha, \beta<p-1$. Let $\mu$ be a positive Borel measure on $[0, 1)$ and $\mu_{\lambda}[z]$ be defined as in (\ref{mu}) for $z >0$. Set $d\nu(t)=(1-t)^{\lambda-1}d \mu(t)$.  If $\nu$ is a  $[1+\frac{1}{p}(\beta-\alpha)]$-Carleson measure on $[0, 1)$, then
\begin{equation}\label{mun-1}\mu_{\lambda}[x+n] \preceq \frac{1}{(x+n)^{1+\frac{1}{p}(\beta-\alpha)}}\end{equation}
holds for all $x>0, n\in \mathbb{N}$. Furthermore, if $\nu$ is a  vanishing $[1+\frac{1}{p}(\beta-\alpha)]$-Carleson measure on $[0, 1)$, then, for any $x>0$,
 \begin{equation}\label{mun-2}\mu_{\lambda}[x+n] =o\left(\frac{1}{(x+n)^{1+\frac{1}{p}(\beta-\alpha)}}\right), \quad n\rightarrow \infty.\end{equation}
\end{lemma}

\begin{proof}
We first consider the case when $x>0, n\geq 2$, we get from integration by parts that
\begin{eqnarray}
\mu_{\lambda}[x+n]&=&\int_{0}^1 t^{x+n-1}d\nu(t) \nonumber \\&=&\nu([0,1))-(x+n-1)\int_{0}^1 t^{x+n-1}\nu([0, t))dt \nonumber \\
&=& (x+n-1)\int_{0}^1 t^{x+n-2}\nu([t, 1))dt.\nonumber
\end{eqnarray}

If $\nu$ is a $[1+\frac{1}{p}(\beta-\alpha)]$-Carleson measure on $[0, 1)$, then we see that there is a constant $C_2>0$ such that
$$\nu([t,1))\leq C_2 (1-t)^{1+\frac{1}{p}(\beta-\alpha)}$$
holds for all $t\in [0,1)$. It follows that
\begin{eqnarray}\mu_{\lambda}[x+n] &\leq & C_2 (x+n-1)\int_{0}^1 t^{x+n-2}(1-t)^{1+\frac{1}{p}(\beta-\alpha)}dt\nonumber \\&=&C_2 \frac{(x+n-1)\Gamma(x+n-1)\Gamma(2+\frac{1}{p}(\beta-\alpha))}{\Gamma(x+n+1+\frac{1}{p}(\beta-\alpha))} .\nonumber \end{eqnarray}

By using (\ref{g}) again, we obtain that
$$\frac{(x+n-1)\Gamma(x+n-1)\Gamma(2+\frac{1}{p}(\beta-\alpha))}{\Gamma(x+n+1+\frac{1}{p}(\beta-\alpha))}  \asymp \frac{1}{(x+n)^{1+\frac{1}{p}(\beta-\alpha)}}.$$

Next we consider the case when $x>0, n=1$.  When $x\geq 1, n=1$, by repeating  the arguments above, we easily see that (\ref{mun-1}) also holds.

When $x\in (0,1), n=1$, we see from, $\nu$ is a finite measure on $[0,1)$, that
\begin{eqnarray}
\mu_{\lambda}[x+1]=\int_{0}^1 t^{x}d\nu(t) \leq \nu([0,1)).\nonumber
\end{eqnarray}

This implies that \begin{eqnarray}
\mu_{\lambda}[x+1] \preceq  \frac{1}{(x+1)^{1+\frac{1}{p}(\beta-\alpha)}}\nonumber
\end{eqnarray}
holds for all $x\in (0, 1)$.
It follows that $$\mu_{\lambda}[x+n] \preceq \frac{1}{(x+n)^{1+\frac{1}{p}(\beta-\alpha)}}$$
holds for all $x>0, n \in \mathbb{N}$.

Similarly, if $\nu$ is a vanishing $[1+\frac{1}{p}(\beta-\alpha)]$-Carleson measure on $[0, 1)$, by minor modifications of above arguments, we can show that (\ref{mun-2}) holds.  The lemma is proved.
\end{proof}

We start to prove Theorem \ref{m-1-3}.
\begin{proof}[Proof of "if" part of Theorem \ref{m-1-3}]
By Lemma \ref{lem} and checking the proof of Theorem \ref{m-1-2}, we see that $ \widehat{\mathbf{H}}^{\alpha, \beta}_{\lambda, \mu}$ is bounded from $L^p(\mathbb{R}_{+})$ into $l^p$, if $d\nu(t)=(1-t)^{\lambda-1}d\mu(t)$ is a $[1+\frac{1}{p}(\beta-\alpha)]$-Carleson measure on $[0, 1)$. The "if" part of Theorem \ref{m-1-3} is proved.
\end{proof}

\begin{proof}[Proof of "only if" part of Theorem \ref{m-1-3}]

In our proof, we need the following well-known estimate, see \cite{Zh}, Page 54.  Let $0<w<1$. For any $c>0$, we have
\begin{equation}\label{est}\sum_{n=1}^{\infty}n^{c-1}w^{2n}\asymp \frac{1}{(1-w^2)^c}.
\end{equation}

For any $0<w<1$, we define
\begin{equation}\label{def}f_w(x):=\begin{cases}
\quad\quad\quad\quad 0, \; \quad \;  \;\,  \text{if} \; x\in (0, 1], \\
(1-w^2)^{\frac{1}{p}}w^{\frac{2(k-1)}{p}}, \; \text{if} \; x\in (k, k+1], k\in \mathbb{N}. \\
\end{cases}
\end{equation}
Then we easily see that
$$\|f_w\|_p^p=(1-w^2)\sum_{k=1}^{\infty}w^{2(k-1)}=1. $$
Also, we set $g_w:=\{g_n\}_{n=1}^{\infty}$ with $g_{n}=(1-w^2)^{\frac{1}{q}}w^{\frac{2(n-1)}{q}}$. Then we have $\|g_w\|_q^q=1.$

In view of the boundedness of $\widehat{\mathbf{H}}^{\alpha, \beta}_{\lambda, \mu}: L^{p}(\mathbb{R}_{+}) \rightarrow l^{p}$, then we get from the duality that
\begin{eqnarray}\label{boun-1}
1 &\succeq & \sum_{n=1}^{\infty} g_n\cdot \widehat{\mathbf{H}}^{\alpha, \beta}_{\lambda, \mu}(f_w)(n) \nonumber \\
&=&(1-w^2)^{\frac{1}{q}}\sum_{n=1}^{\infty}n^{\frac{\beta}{p}}w^{\frac{2(n-1)}{q}}\cdot \int_{\mathbb{R}_{+}}\left[x^{-\frac{\alpha}{p}}f_{w}(x)\int_{0}^1t^{x+n-1}d\nu(t)\right]\,dx \nonumber \\
&\geq  &(1-w^2)^{\frac{1}{q}}\sum_{n=1}^{\infty}n^{\frac{\beta}{p}}w^{\frac{2(n-1)}{q}}\cdot \int_{\mathbb{R}_{+}}\left[x^{-\frac{\alpha}{p}}f_{w}(x)\int_{w}^1t^{x+n-1}d\nu(t)\right]\,dx.
\end{eqnarray}

We notice that
\begin{equation}\label{est-1}t^{x}\geq t^{k+1}, {\text {when}} \,\, x\in (k, k+1]\end{equation}
holds for all $ k\in \mathbb{N}, t\in (0,1)$, and
\begin{equation}\label{est-2}x^{-\frac{\alpha}{p}}\asymp k^{-\frac{\alpha}{p}}\end{equation}
holds for $x\in (k, k+1]$, $k\in \mathbb{N}$.

It follows that
\begin{eqnarray}\label{boun-1-2}
1 &\succeq  &(1-w^2)^{\frac{1}{q}}\sum_{n=1}^{\infty}n^{\frac{\beta}{p}}w^{\frac{2(n-1)}{q}}\cdot \sum_{k=1}^{\infty }\left[k^{-\frac{\alpha}{p}}(1-w^2)^{\frac{1}{p}}w^{\frac{2(k-1)}{p}}\int_{w}^1t^{k+n}d\nu(t)\right]\nonumber\\
&\geq  & (1-w^2)\nu([w,1))\sum_{n=1}^{\infty}n^{\frac{\beta}{p}}w^{\frac{2(n-1)}{q}}\cdot \sum_{k=1}^{\infty }\left[k^{-\frac{\alpha}{p}}w^{\frac{2(k-1)}{p}}w^{k+n}\right] \nonumber \\
&\geq & (1-w^2)\nu([w,1))\left[\sum_{n=1}^{\infty}n^{\frac{\beta}{p}}w^{(\frac{2}{q}+1)n}\right]\cdot \left[\sum_{k=1}^{\infty }k^{-\frac{\alpha}{p}}w^{(\frac{2}{p}+1)k}\right] \nonumber
\end{eqnarray}

Thus, by  (\ref{est}), we have
\begin{eqnarray}\label{boun-2}
1 &\succeq & (1-w^2)\nu([w,1))\cdot\frac{1}{(1-w^2)^{1+\frac{\beta}{p}}}\cdot \frac{1}{(1-w^2)^{1-\frac{\alpha}{p}}}\nonumber \end{eqnarray}

This implies that
$$\nu([w, 1))\preceq (1-w^2)^{1+\frac{1}{p}(\beta-\alpha)}$$
for all $0<w<1$. It follows that $\nu$ is a $[1+\frac{1}{p}(\beta-\alpha)]$-Carleson measure on $[0, 1)$ and the "only if" part is proved.
\end{proof}

We next prove Theorem \ref{m-1-4}.  We first show the "if" part of Theorem \ref{m-1-4}. Let $\mathfrak{N}\in \mathbb{N}$. We define the operator $\mathbf{H}^{[\mathfrak{N}]}$ as, for $f\in \mathcal{M}(\mathbb{R}_{+})$, $$\mathbf{H}^{[\mathfrak{N}]}(f)(n):=n^{\frac{\beta}{p}}\int_{\mathbb{R}_{+}}x^{-\frac{\alpha}{p}}\mu_{\lambda}[x+n]f(x)dx,$$
when $n\leq \mathfrak{N}$, and $\mathbf{H}^{[\mathfrak{N}]}(f)(n):=0,$ when $n\geq \mathfrak{N}+1$.

We see that $\mathbf{H}^{[\mathfrak{N}]}$ is a finite rank operator and hence it is compact from $L^{p}(\mathbb{R}_{+})$ into $l^{p}$.

By Lemma \ref{lem}, we see that, for any $\epsilon>0$, there is an $\widehat{\mathbf{N}}\in \mathbb{N}$ such that
$$\mu_{\lambda}[x+n] \preceq \frac{\epsilon}{(x+n)^{1+\frac{1}{p}(\beta-\alpha)}}$$
holds for all $x>0, n>\widehat{\mathbf{N}}$.

Then, we see from
\begin{eqnarray}\|(\widehat{\mathbf{H}}^{\alpha, \beta}_{\lambda, \mu}-\mathbf{H}^{[\mathfrak{N}]})f\|_{p}^p=\sum_{n=\mathfrak{N}+1}^{\infty}n^{\frac{\beta}{p}}\int_{\mathbb{R}_{+}}x^{-\frac{\alpha}{p}}\mu_{\lambda}[x+n]f(x)dx.\nonumber
\end{eqnarray}
that,
\begin{eqnarray}\|(\widehat{\mathbf{H}}^{\alpha, \beta}_{\lambda, \mu}-\mathbf{H}^{[\mathfrak{N}]})f\|_{p}^p\preceq
 {\epsilon}^p \sum_{n=\mathfrak{N}+1}^{\infty}n^{\frac{\beta}{p}}\int_{\mathbb{R}_{+}}\left|\frac{f(x)}{(x+n)^{1+\frac{1}{p}(\beta-\alpha)}}dx\right|^p,\nonumber\end{eqnarray}
when $\mathfrak{N}>\widehat{\mathbf{N}}.$

Consequently, by checking the proof of Theorem \ref{m-1-2},  we see that, for any $\epsilon>0$, it holds that
\begin{eqnarray}\|(\widehat{\mathbf{H}}^{\alpha, \beta}_{\lambda, \mu}-\mathbf{H}^{[\mathfrak{N}]})f\|_{p}\preceq
 {\epsilon}  \|f\|_p\nonumber\end{eqnarray}
for all $\mathfrak{N}>\widehat{\mathbf{N}}$. It follows that $\widehat{\mathbf{H}}^{\alpha, \beta}_{\lambda, \mu}$ is compact from $L^{p}(\mathbb{R}_{+})$ to $l^p$. This proves the "if" part.

Finally, we show the "only if" part.  For $0<w<1$.  We take $f_w$ as be defined in (\ref{def}). It is easy to check that ${f}_w$ is convergent weakly to $0$ in $L^{p}(\mathbb{R}_{+}).$

Since $\widehat{\mathbf{H}}^{\alpha, \beta}_{\lambda, \mu}$ is compact from $L^{p}(\mathbb{R}_{+})$ to $l^p$, we get
\begin{equation}\label{com-0}\lim_{w\rightarrow {1^{-}}} \|\widehat{\mathbf{H}}^{\alpha, \beta}_{\lambda, \mu}{f}_w)\|_{p}=0.\end{equation}

On the other hand, we have
\begin{eqnarray}
\|\widehat{\mathbf{H}}^{\alpha, \beta}_{\lambda, \mu}{f}_w)\|_{p}^p
&=&\sum_{n=1}^{\infty} n^{\beta} \left\{ \int_{\mathbb{R}_{+}}\left[x^{-\frac{\alpha}{p}}f_{w}(x)\int_{0}^1t^{x+n-1}d\nu(t)\right]\,dx\right\}^p\nonumber   \\
&\geq & \sum_{n=1}^{\infty} n^{\beta}\left\{ \int_{\mathbb{R}_{+}}\left[x^{-\frac{\alpha}{p}}f_{w}(x)\int_{w}^1t^{x+n-1}d\nu(t)\right]\,dx\right\}^p.\nonumber
\end{eqnarray}
It follows from (\ref{est-1}) and (\ref{est-2}) that
\begin{eqnarray}
\|\widehat{\mathbf{H}}^{\alpha, \beta}_{\lambda, \mu}{f}_w)\|_{p}^p
&\succeq&(1-w^2)\sum_{n=1}^{\infty}n^{\beta}\left\{ \sum_{k=1}^{\infty}\left[k^{-\frac{\alpha}{p}}w^{\frac{2(k-1)}{p}}\int_{w}^1t^{k+n}d\nu(t)\right]\right\}^p\nonumber
\end{eqnarray}

Consequently,  we obtain that
\begin{eqnarray}
\|\widehat{\mathbf{H}}^{\alpha, \beta}_{\lambda, \mu}{f}_w)\|_{p}^p
&\succeq&(1-w^2)[\nu([w,1))]^p\left[\sum_{n=1}^{\infty}n^{\beta}w^{pn}\right] \left[\sum_{k=1}^{\infty}k^{-\frac{\alpha}{p}}w^{\frac{2(k-1)}{p}+k}\right] ^p\nonumber
\end{eqnarray}
Then, by again (\ref{est}), we get that
\begin{eqnarray}
\|\widehat{\mathbf{H}}^{\alpha, \beta}_{\lambda, \mu}{f}_w)\|_{p}^p
&\succeq&(1-w^2)[\nu([w,1))]^p\cdot \frac{1}{(1-w^2)^{1+\beta}}\cdot \frac{1}{(1-w^2)^{p-\alpha}}.\nonumber
\end{eqnarray}

This means that
\begin{eqnarray}
\nu([w,1))\preceq \|\widehat{\mathbf{H}}^{\alpha, \beta}_{\lambda, \mu}{f}_w)\|_{p} (1-w^2)^{1+\frac{1}{p}(\beta-\alpha)}.\nonumber
\end{eqnarray}
It follows from (\ref{com-0}) that $\nu$ is a vanishing $[1+\frac{1}{p}(\beta-\alpha)$-Carleson measure on $[0, 1)$. This proves the "only if" part of Theorem \ref{m-1-4}  and the proof of Theorem \ref{m-1-4} is completed.

\section{Final Remarks}

\begin{remark} We first remark that the assumptions $-1<\alpha, \beta<p-1$ in Theorem \ref{m-1-1} and \ref{m-1-2} are both necessary.

We consider the case when $\theta_1=\theta_2=\lambda=1, \alpha=\beta:=\gamma$. That is to say, we will consider the operator
$$\mathbf{H}^{\gamma, \gamma }_{1, 1, 1}(f)(n)=n^{\frac{\gamma}{p}}\int_{\mathbb{R}_{+}}
\frac{x^{-\frac{\gamma}{p}}}{x+n}f(x)\,dx, f \in \mathcal{M}(\mathbb{R}_{+}), n \in \mathbb{N}.$$

We will write $\widehat{\mathbf{H}}_{\gamma}$ to denote $\mathbf{H}^{\gamma, \gamma }_{1, 1, 1}$. We shall show that

\begin{proposition}$\widehat{\mathbf{H}}_{\gamma}$ is not bounded from $L^{p}(\mathbb{R}_{+})$ into $l^p$, if $\gamma\leq -1$, or $\gamma\geq p-1$.\end{proposition}
\begin{proof}
For $\varepsilon>0$, we take
\begin{equation*}\widehat{f}(x):=\begin{cases}
\quad\quad 0, \; \; \; \text{if} \; x\in (0, 1], \\
\varepsilon^{\frac{1}{p}} x^{-\frac{1+\varepsilon}{p}},\; \text{if} \; x\in [1, \infty). \\
\end{cases}
\end{equation*}
Then, we easily see that $\|\widehat{f}\|_p^p=1,$ and
$$\|\widehat{\mathbf{H}}_{\gamma}\widehat{f}\|_p^p=\sum_{n=1}^{\infty}n^{\gamma}\left[\int_{1}^{\infty}
\frac{x^{-\frac{1+\gamma+\varepsilon}{p}}}{x+n}\,dx\right]^p. $$

(1) If $\gamma<-1$,  when $\varepsilon<-(\gamma+1)$, we see from $1+\gamma+\varepsilon<0$ that, for a fixed $n \geq 1$, it holds that
$$\int_{1}^{\infty}
\frac{x^{-\frac{1+\gamma+\varepsilon}{p}}}{x+n}\,dx\geq \int_{n}^{\infty}
\frac{x^{-\frac{1+\gamma+\varepsilon}{p}}}{x+n}\,dx\geq  \frac{1}{2} \int_{n}^{\infty}
x^{-1-\frac{1+\gamma+\varepsilon}{p}}\,dx=+\infty. $$
This means that $\widehat{\mathbf{H}}_{\gamma}$ is not bounded from $L^{p}(\mathbb{R}_{+})$ into $l^p$ in this case.

(2)  If $\gamma=-1$ or $\gamma \geq p-1$, we have
\begin{eqnarray}\label{w-1}\|\widehat{\mathbf{H}}_{\gamma}\widehat{f}\|_p^p&=&\sum_{n=1}^{\infty}n^{-\gamma}\left[\int_{1}^{\infty}
\frac{x^{-\frac{1+\gamma+\varepsilon}{p}}}{x+n}\,dx\right]^p=\sum_{n=1}^{\infty}n^{-1-\varepsilon}\left[\int_{\frac{1}{n}}^{\infty}
\frac{t^{-\frac{1+\gamma+\varepsilon}{p}}}{1+t}\,dt\right]^p\nonumber  \\
&\geq &\nonumber \sum_{n=1}^{\infty}n^{-1-\varepsilon}\left[\int_{1}^{\infty}
\frac{t^{-\frac{1+\gamma+\varepsilon}{p}}}{1+t}\,dt\right]^p \end{eqnarray}

We note  that
\begin{eqnarray}
\sum_{n=1}^{\infty}n^{-1-\varepsilon}=\frac{1}{\varepsilon}[1+o(1)], \, \varepsilon \rightarrow 0^{+}.\nonumber
\end{eqnarray}

Therefore, we conclude that
\begin{eqnarray}\|\widehat{\mathbf{H}}_{\gamma}\widehat{f}\|_p^p\geq  \frac{1}{\varepsilon}[1+o(1)]\left[\int_{1}^{\infty}
\frac{t^{-\frac{1+\gamma+\varepsilon}{p}}}{1+t}\,dt\right]^p.\nonumber \end{eqnarray}
Take $\varepsilon \rightarrow 0^{+}$, we get that $\|\widehat{\mathbf{H}}_{\gamma}\widehat{f}\|_p^p \rightarrow +\infty$.  This implies that $\widehat{\mathbf{H}}_{\gamma}: L^p(\mathbb{R}_{+}) \rightarrow l^p$ is not bounded when $\gamma=-1$ or $\gamma\geq p-1$.
The proposition is proved.
\end{proof}
\end{remark}

\begin{remark}
Take $\lambda=1$ in Theorem \ref{m-1-3}, we obtain that
\begin{corollary}
Let $p>1$, $-1<\alpha, \beta<p-1$.  Let $\mu$ be a finite positive Borel measure on $[0, 1)$. Then $\widehat{\mathbf{H}}_{\mu}^{\alpha, \beta} : L^p(\mathbb{R}_{+})\rightarrow l^p$ is bounded if and only if
$\mu$ is a $[1+\frac{1}{p}(\beta-\alpha)]$-Carleson measure on $[0, 1)$.
Here $$\widehat{\mathbf{H}}^{\alpha, \beta}_{\mu}(f)(n):=n^{\frac{\beta}{p}}\int_{\mathbb{R}_{+}}x^{-\frac{\alpha}{p}}\mu[x+n]f(x)dx,\, f\in  \mathcal{M}(\mathbb{R}_{+}), \, n\in \mathbb{N},$$
and
\begin{equation*}\mu[z]:=\int_{[0, 1)}t^{z-1}d\mu(t), \, z>0.\end{equation*}
\end{corollary}

We notice that Proposition 2.1 in the recent work \cite{BSW} of Bao et. al. implies that
\begin{proposition}\label{las}
Let $r\in \mathbb{R}, s>0$ with $s>|r|$. Let  $\mu, \nu$ be two finite positive Borel measures on $[0, 1)$ with $d\nu(t)=(1-t)^{r}d\mu(t)$. Then
$\nu$ is a $(s+r)$-Carleson measure on  $[0, 1)$ if and only if $\mu$ is a $s$-Carleson measure on  $[0, 1)$.
\end{proposition}

On the other hand, for $p>1$, $\lambda>0$, $-1<\alpha, \beta<p-1$,  if $\alpha\leq \beta$ and $0<\lambda<2+\frac{1}{p}(\beta-\alpha)$, or $\alpha>\beta$ and $-\frac{1}{p}(\beta-\alpha)<\lambda<2+\frac{1}{p}(\beta-\alpha)$,  then it holds that $|\lambda-1|<1+\frac{1}{p}(\beta-\alpha)$.

Hence, from Proposition \ref{las} and  Theorem \ref{m-1-3}, we see that
\begin{corollary}
Let $p>1$, $\lambda>0$, $-1<\alpha, \beta<p-1$.  Let  $\mu, \nu$ be two finite positive Borel measures on $[0, 1)$ with $d\nu(t)=(1-t)^{\lambda-1}d\mu(t)$.

If $\alpha\leq \beta$ and $0<\lambda<2+\frac{1}{p}(\beta-\alpha)$, or $\alpha>\beta$ and $-\frac{1}{p}(\beta-\alpha)<\lambda<2+\frac{1}{p}(\beta-\alpha)$.
Then $\widehat{\mathbf{H}}^{\alpha, \beta}_{\lambda, \mu} : L^p(\mathbb{R}_{+})\rightarrow l^p$ is bounded if and only if one of the following two conditions is satisfied.

(1) $\mu$ is a $[2+\frac{1}{p}(\beta-\alpha)-\lambda]$-Carleson measure on $[0, 1)$.

(2) $\nu$ is a $[1+\frac{1}{p}(\beta-\alpha)]$-Carleson measure on $[0, 1)$.
\end{corollary}
\end{remark}


\begin{remark}

Let $\mathbb{D}$ be the unit disk in the complex plane $\mathbb{C}$.  We denote by $\mathcal{H}(\mathbb{D})$ the class of all analytic functions on $\mathbb{D}$.  Let $0<p<\infty$, the Hardy space $H^p(\mathbb{D})$ is the class of all $f\in \mathcal{H}(\mathbb{D})$  such that
$$\|f\|_{H^p}=\sup_{r\in (0, 1)}M_{p}(r, f)<\infty,$$
where $$M_{p}(r ,f)=\{\frac{1}{2\pi}\int_{0}^{2\pi}|f(re^{i\theta})|^p\,d\theta\}^{\frac{1}{p}},\:0<r<1.$$
It is well known that a function $f(z)=\sum_{n=0}^{\infty}a_n z^n \in \mathcal{H}(\mathbb{D})$ belongs to  $H^2(\mathbb{D})$ if and only if $\sum_{n=0}^{\infty}|a_n|^2<+\infty$.
See \cite{D} for more introductions to the theory of Hardy spaces.

We define
$$\mathcal{H}(f)(z):=\sum_{n=0}^{\infty}\left [\int_{\mathbb{R}_{+}}\frac{f(x)}{x+n+1}\,dx\right]z^n, \, f\in \mathcal{M}(\mathbb{R}_{+}), z\in \mathbb{C}.$$
Then, we see from Theorem \ref{m-1-1} that $\mathcal{H}$ is bounded from $L^{2}(\mathbb{R}_{+})$ into $H^2(\mathbb{D})$.

It is natural to consider the following
\begin{question}
Whether $\mathcal{H}$ is bounded from $L^{p}(\mathbb{R}_{+})$ into $H^p(\mathbb{D})$ for any $p>0$?
\end{question}
\end{remark}






\end{document}